\documentclass[10pt]{article}
\title{The Connes--Kirchberg Problem and infinite-dimensional phenomena in quantum information theory}
\author{Magdalena Musat \thanks{This work was supported by a research grant from the Danish Council for Independent Research, Natural Sciences.}}

\usepackage[english]{babel}
\usepackage{upref,amsfonts,amsxtra,latexsym,color}
\usepackage{amsmath,amsthm,epsf,mathrsfs,verbatim,hyperref}
\usepackage[all]{xy}
\usepackage{titlesec}

\usepackage{calc}
\usepackage{fp}
\usepackage{amssymb}
\newcommand{\Id}{\mathrm{id}}

\newcommand{\Tr}{\mathrm{Tr}}
\newcommand{\tr}{\mathrm{tr}}
\newcommand{\fin}{\mathrm{fin}}
\newcommand{\mat}{\mathrm{matr}}
\newcommand{\UCPT}{\mathrm{UCPT}}
\newcommand{\Aut}{\mathrm{Aut}}
\newcommand{\FM}{\mathcal{FM}}

\newcommand{\cM}{{\mathcal M}}
\newcommand{\ep}{{\varepsilon}}
\newcommand{\cN}{{\mathcal N}} 
\newcommand{\cA}{{\mathcal A}}
\newcommand{\cG}{{\mathcal G}}
\newcommand{\cD}{{\mathcal D}}

\newcommand{\cU}{{\mathcal U}}

\newcommand{\cB}{{\mathcal B}}

\newcommand{\cR}{{\mathcal R}}

\newcommand{\C}{{\mathbb C}}
\newcommand{\Z}{{\mathbb Z}}
\newcommand{\N}{{\mathbb N}}

\newcommand{\Q}{{\mathbb Q}}

\newcommand{\R}{{\mathbb R}}
\newcommand{\Cs}{{$C^*$-al\-ge\-bra}}
\newcommand{\conv}{{\mathrm{conv}}}
\newcommand{\sh}{{$^*$-ho\-mo\-mor\-phism}}

\theoremstyle{plain}
\newtheorem{theorem}[subsection]{Theorem}

\newtheorem{corollary}[subsection]{Corollary}
\newtheorem{proposition}[subsection]{Proposition}
\theoremstyle{definition}
\newtheorem{definition}[subsection]{Definition}
\newtheorem{example}[subsection]{Example}
\newtheorem{conjecture}[subsection]{Conjecture}

\newtheorem{remark}[subsection]{Remark}

\titleformat{\section}
  {\normalfont\scshape}{\thesection}{1em}{}

\date{}

\begin{document}

\maketitle 

\vspace{-.5cm}

\begin{center} \emph{In memory of Eberhard Kirchberg}
\end{center}

\begin{abstract}
We give an overview of results tying together a circle of problems connected to the Connes Embedding Problem, Kirchberg's reformulations thereof, Tsirelson's conjecture and its relation to quantum information theory, and a class of quantum channels, called factorizable, introduced by Anantharaman-Delaroche. While parts of the article are more expository, there are new results, including obstructions for channels to being $k$-noisy (admitting a factorization through a full matrix algebra).
\end{abstract}

\section{Introduction} \label{sec:intro}

\noindent The Connes Embedding Problem (CEP), originating in Connes' famous classification paper, \cite{Con:class}, from 1976, is likely the most influential question in operator algebras over the last 50 years, which has resulted in a wealth of interconnections between different fields of mathematics, including, besides operator algebras, quantum information theory, geometric group theory, and lately also complexity theory. Kirchberg's deep recasting of the CEP in terms of uniqueness of the $C^*$-tensor norm on the algebraic tensor product of the full group \Cs{} of the free group with itself, correlation matrices of unitaries in tracial von Neumann algebras, and his famous QWEP conjecture, all from his 1993 Inventiones paper, \cite{Kir:CEP}, are milestones. Kirchberg's results became the bridge between the CEP and Tsirelson's conjecture in quantum information theory. It was via Tsirelson's conjecture that the Connes Embedding Problem recently has found its negative solution in the paper MIP$^*$=RE, \cite{JNVWY:MIP*=RE}, by Ji--Natarajan--Vidick--Wright--Yuen, using complexity theory.

Anantharaman-Delaroche introduced in \cite{A-D:factorizable} a class of unital completely  positive maps between matrix algebras (or more generally between von Neumann algebras) that admit a factorization through a von Neumann algebra with a distinguished normal faithful state, that often is taken to be tracial. These maps are called {\em factorizable}. Although originally aimed at studying ergodic actions of free groups on noncommutative spaces in the more general set-up, factorizable channels between matrix algebras have found important applications within quantum information theory. They were in \cite{HaaMusat:CMP-2011} used to find counterexamples to the asymptotic Quantum Birkhoff Conjecture. From the point of view of this paper, factorizable maps provide 
 interesting reformulations of the CEP in terms of their infinite-dimensional behavior, and Kirchberg's correlation matrices of unitaries in tracial von Neumann algebras can be used to understand factorizable Schur multipliers.

We shall in this paper give an overview of results tying together these diverse topics, and we also include some so far unpublished results about factorizable maps, including Proposition~\ref{prop:k-noisy} about factorizable channels (in fact, mixtures of unitaries) which are not $k$-noisy (do not factor via a full matrix algebra), properties of factorizable channels that factor through matrix algebras versus those that factor through finite-dimensional von Neumann algebras,  Proposition~\ref{prop:FM}, and some new examples showing that the ancilla of a factorizable channel is far from unique, Example~\ref{ex:depolarizing}. We also provide a self-contained proof of Tsirelson's theorem (in the finite- dimensional case), based on \Cs{} techniques. 

In Section~\ref{sec:CEP} we discuss the Connes Embedding Problem and Kirchberg's reformulation thereof in terms of correlation matrices of unitaries, considered by Dykema--Juschenko in \cite{DykJus:moments}. Section~\ref{sec:factorizable} is about factorizable quantum channels and contains an overview of some of the most important results about these, including results about factorizable channels that cannot be expressed with finite dimensional (or type I) ancillas, as well as some so far unpublished results mentioned above. Section~\ref{sec:Tsirelson} gives a short overview of quantum correlation matrices, Tsirelson's conjecture, Kirchberg's reformulations of the Connes Embedding Problem, and a short account on how this led to establishing the equivalence between the CEP and Tsirelson's conjecture. 

\section{The Connes Embedding Problem and correlation matrices of unitaries} \label{sec:CEP}

\noindent In his seminal classification paper from 1976, \cite{Con:class}, as a bi-product A. Connes shows that the II$_1$-factors associated with the free groups embed into the ultrapower $\cR^\omega$ of the hyperfinite type II$_1$-factor $\cR$. He remarks that ``apparently such an imbedding ought to exist for all II$_1$ factors because it does for the regular representation of free groups''. This remark is now known as the \emph{Connes Embedding Problem}, abbreviated CEP, and more precisely asks if every II$_1$-factor with separable predual embeds into $\cR^\omega$. This seemingly off-hand remark of A. Connes became arguably the most important problem in operator algebras, with numerous interconnections and applications in Neumann algebras, \Cs s, geometric group theory, quantum information theory and beyond.  In \cite[Lemma 5.17+Lemma 5.22]{Con:class}, Connes restates the embedding problem as follows:

\begin{proposition}[Connes] \label{prop:CEP} 
Let $\cN$ be a finite factor with separable predual. Then $\cN$ admits a normal embedding into $\cR^\omega$ if and only if for all integers $n,\ell \ge 0$, all unitaries $u_1, \dots, u_n$ in $\cN$, and all $\ep >0$, there exist $k \in \N$ and unitaries $v_1, \dots, v_n$ in $M_k(\C)$ such that
$$\big| \tau_\cN(u_{i_1}^{\nu_1} \cdots u_{i_r}^{\nu_r}) - \tr_k(v_{i_1}^{\nu_1} \cdots v_{i_r}^{\nu_r}) \big| < \ep,$$
for all $r \ge 1$, $i_1, \dots, i_r \in \{1, \dots, n\}$, and $\nu_1, \dots, \nu_r \in \Z$ with $\sum |\nu_j| \le \ell$. 
\end{proposition}

\noindent The conclusion of the proposition is also valid for finite factors that do not have separable predual, so we can rephrase CEP as the question if all tracial von Neumann algebras satisfy the conclusion of Proposition~\ref{prop:CEP}. This can, informally, be recast as the question if any finite factor \emph{approximately} embeds into a matrix algebra.

Connes actually only proved the ``if'' part of Proposition~\ref{prop:CEP} in \cite{Con:class}. The proof of the ``only if'' part consists of a number of approximation steps: Given unitaries $u_1, \dots, u_n$ in $\cN \subseteq \cR^\omega$, we first find unitaries $w_1, \dots, w_n$ in $\cR$ so that $\tau_\cR(w_{i_1}^{\nu_1} \cdots w_{i_r}^{\nu_r})$ is sufficiently close to $\tau_\cN(u_{i_1}^{\nu_1} \cdots u_{i_r}^{\nu_r})$ for all words of length at most $\ell$. The idea behind this step is that any unitary $u$ in $\cR^\omega$ lifts to a unitary $\{u_n\}_n$ in $\ell^\infty(\N,\cR)$, and $\tau_{\cR^\omega}(u) = \lim_\omega \tau_\cR(u_n)$ by the construction of ultrapowers. 
The hyperfinite II$_1$-factor $\cR$ contains a UHF-algebra $\cA$ as a strong operator dense sub-algebra, so by Kaplansky's theorem we can approximate $w_1, \dots, w_n$ in the strong operator topology (and hence with respect to the tracial 2-norm)   with unitaries $z_1, \dots, z_n$ in $\cA$. Finally, the unitaries $z_1, \dots, z_n$ can be approximated in norm by unitaries $v_1, \dots, v_n$ residing in a matrix algebra contained in $\cA$. Making sure that each approximation is fine enough, we get the desired estimates. 

In \cite{Kir:CEP}, E. Kirchberg reformulated CEP in several ways, importantly connecting it to Tsirelson's conjecture, which in turn led to its negative solution by Ji, Natarajan, Vidick, Wright and Yuen in \cite{JNVWY:MIP*=RE}, see more in Section~\ref{sec:Tsirelson} below. One of Kirchberg's deep observations in \cite{Kir:CEP} is that one only needs to consider second order moments in Proposition~\ref{prop:CEP} above:

\begin{theorem}[Kirchberg] \label{thm:Kir1} Let $\cN$ be a finite factor with separable predual. It follows that $\cN$ embeds into $\cR^\omega$ if and only if for each integer $n \ge 2$, all unitaries $u_1, \dots, u_n$ in $\cN$, and all $\ep >0$ there exist $k \in \N$ and unitaries $v_1, \dots, v_n$ in $M_k(\C)$ such that
$$\big| \tau_\cN(u_iu_j^*) - \tr_k(v_iv_j^*) \big| < \ep,$$
for all $i,j = 1, \dots, n$. 
\end{theorem}

\noindent This theorem is Proposition 4.6 in \cite{Kir:CEP}. For completeness of this exposition and to provide the reader with more details than given in \cite{Kir:CEP}, we include here a proof of this beautiful result. 

\begin{proof} The ``only if'' part follows from Proposition~\ref{prop:CEP}. The first step of the proof of the ``if'' part is to establish the existence of a  $\| \, \cdot \, \|_2$-isometric unital linear map $\Phi \colon \cN \to \cR^\omega$. By the separability assumption we can find a sequence $\{u_1, u_2, \dots \}$ of unitaries in $\cN$ that is $\| \, \cdot \, \|_2$-dense in the unitaries in  $\cN$ and with $u_1=1$. As all matrix algebras embed into $\cR$, we can for each $n \ge 1$ find unitaries $v_{n,1}, \dots, v_{n,n}$ in $\cR$ with $v_{n,1}=1_{\cR}$ and 
$$\big| \tau_\cN(u_iu_j^*) - \tau_\cR(v_{n,i}v_{n,j}^*)\big| < 1/n,$$
for all $i,j = 1, \dots, n$. (Note that we always can take $v_1=1_{\cR}$ in the theorem, upon replacing $v_i$ with $v_iv_1^*$.) Set $v_{n,j} = 1_{\cR}$, for $j >n$, and set $v_j = [\{v_{n,j}\}_n] \in \cR^\omega$, where $[ \, \cdot \, ]$ denotes the quotient mapping $\ell^\infty(\N,\cR) \to \cR^\omega$. It then follows that $v_1, v_2, \dots $ are unitaries in $\cR^\omega$ satisfying $v_1=1_{\cR}$ and $\langle u_i,u_j \rangle_{\tau_\cN} = \tau_{\cN}(u_iu_j^*) = \tau_{\cR^\omega}(v_iv_j^*) = \langle v_i,v_j \rangle_{\tau_{\cR^\omega}}$, for all $i,j \in \N$. 

A simple linear algebra argument shows that the map $u_i \mapsto v_i$ extends to a well-defined unital isometric linear map $\Phi_0 \colon \mathrm{span}\{u_1,u_2, \dots\} \to \cR^\omega$. 
Extend $\Phi_0$ isometrically to a Hilbert space isometry $\bar{\Phi} \colon L^2(\cN,\tau_\cN) \to L^2(\cR^\omega,\tau_{\cR^\omega})$, viewing $\cN$ as a dense subspace of $L^2(\cN,\tau_\cN)$. As the set $\{u_1, u_2, \dots \}$  is $\| \, \cdot \, \|_2$-dense in the unitaries in $\cN$, and since unitaries in a finite von Neumann algebra are closed with respect to $\| \, \cdot \, \|_2$, the isometry $\bar{\Phi}$ maps unitaries in $\cN$ into unitaries in $\cR^\omega$. The restriction, $\Phi$, of $\bar{\Phi}$ to $\cN$ is therefore the desired isometry from $\cN$ into $\cR^\omega$, which moreover maps unitaries in $\cN$ into unitaries in $\cR^\omega$. 

Each contraction in a finite factor is the mean of two unitaries, so we can further conclude that $\Phi$ is norm contractive (as it maps unitaries to unitaries), and hence positive, being unital. In particular, $\Phi$ maps symmetries to symmetries and therefore projections to projections. This entails that $\Phi$ is a Jordan $^*$-homomorphism. Indeed, $\Phi$ is a Jordan $^*$-homomorphism on its multiplicative domain $\{x \in \cN_{\mathrm{sa}} : \Phi(x)^2=\Phi(x^2)\}$,  which is a (norm-closed) Jordan algebra. By the observation above,  the multiplicative domain contains all projections in $\cN$, and each self-adjoint element of $\cN$ is a norm limit of linear combinations of projections in $\cN$, so the multiplicative domain must be  all of $\cN_\mathrm{sa}$. 

Finally, by a theorem of St\o rmer, \cite[Theorem 3.3]{Stormer:Jordan}, we can write $\Phi = \Phi_1 +\Phi_2$ as a sum of a $^*$-homomorphism $\Phi_1$ and  a $^*$-anti-homomorphism $\Phi_2$. Specifically, we can take  $\Phi_1(x) = \Phi(x)p$ and $\Phi_2(x) = \Phi(x)(1_{\cR^\omega}-p)$, for $x \in \cN$, for some central projection $p$ in the strong operator closure of the image of $\Phi$.  Since $\cR$ is $^*$-anti-isomorphic to itself, so is $\cR^\omega$ and its corners. It follows that $\Psi = \Phi_1 + \sigma \circ \Phi_2$ is a  unital $^*$-homomorphism  from $\cN$ into  $\cR^\omega$, where $\sigma$ is a  $^*$-anti-isomorphism of $(1_{\cR^\omega}-p)\cR^\omega(1_{\cR^\omega}-p)$. Finally, as $\Psi$ necessarily is trace preserving, it is automatically normal, so the image of $\Psi$ is a sub-von Neumann algebra of $\cR^\omega$ isomorphic to $\cN$. 
\end{proof}

\noindent This theorem of Kirchberg led Dykema and Juschenko, \cite{DykJus:moments}, to consider certain sets of matrices of correlations. In the following we redefine these sets slightly. Namely,  for $n,k \ge 2$, set
\begin{eqnarray*}
\cG_k(n) &\!=\!& \big\{\big[\tr_k(u_iu_j^*)\big] : u_1, \dots, u_n \in \cU(M_k(\C))\big\}, \; \; \;
\cG_\mat(n) \, = \, \textstyle{\bigcup_{k=1}^\infty} \cG_k(n)\\
\cG_\fin(n)  &\!=\!&  \big\{\big[\tau(u_iu_j^*)\big] : u_1, \dots, u_n \in \cU(\cN), \; \text{$(\cN,\tau)$ is a tracial} \\ && \text{finite-dimensional von Neumann algebra}\big\},\\
\cG_I(n)  &\!=\!&  \big\{\big[\tau(u_iu_j^*)\big] : u_1, \dots, u_n \in \cU(\cN), \; \text{$(\cN,\tau)$ is a tracial von Neumann} \\ && \text{algebra of type I}\big\},\\
\cG(n)  &\!=\!&  \big\{\big[\tau(u_iu_j^*)\big] : u_1, \dots, u_n \in \cU(\cN), \; \text{$(\cN,\tau)$ is any tracial  von}\\&& \text{Neumann algebra}\big\}.
\end{eqnarray*}
Here $\cU(\cN)$ denotes the group of unitary elements in $\cN$. In the definitions of $\cG_\fin(n)$, $\cG_I(n)$, and $\cG(n)$ we allow the tracial state $\tau$ to vary. By a tracial von Neumann algebra $(\cN,\tau)$ we mean a von Neumann algebra $\cN$ equipped with a normal faithful tracial state $\tau$. All tracial states on finite-dimensional von Neumann algebras are automatically normal. Any normal tracial state is faithful on a central summand and zero on its complement, so the faithfulness assumption is not essential.  In the definition of $\cG(n)$ we can replace the tracial von Neumann algebra with the more general situation of a unital \Cs{} equipped with a tracial state, since its enveloping von Neumann algebra via the GNS-representation with respect to the tracial state then is a tracial von Neumann algebra. Further, each von Neumann algebra with a normal faithful tracial state embeds in a trace preserving way into a II$_1$-factor, so it suffices to consider II$_1$-factors in the definition of $\cG(n)$. Using an ultrapower argument, one can see that $\cG(n)$ is closed, for all $n \ge 1$. (By closed, here we mean relatively to $M_n(\C)$. Since  $\cG(n)$ is also bounded, this is the same as saying that $\cG(n)$ is compact.)

\begin{remark} We summarize here some relations between the sets defined above. For $n\geq 2$, let $\Theta(n)$ be the set of $n \times n$ matrices of correlations (i.e., positive definite matrices with diagonal entries equal to 1). Part (vi) is Theorem~\ref{thm:Kir1}, while (iii) and (iv) were proved by Dykema and Juschenko in \cite{DykJus:moments}:
\begin{enumerate}
\item $\cG_k(n) \subseteq \cG_{k'}(n)$ if $k | k'$,
\item $\cG_\mat(n) \subseteq \cG_\fin(n) \subseteq \cG_I(n) \subseteq \cG(n) \subseteq \Theta(n)$,
\item $\cG_\mat(3) = \cG_\fin(3) = \cG_I(n) = \cG(3) = \Theta(3)$, 
\item $\cG(n) \subsetneq \Theta(n)$, for $n \ge 4$. 
\item $\overline{\cG_\mat(n)} = \overline{\cG_\fin(n)}$ and $\cG_\fin(n) = \conv \,  \cG_\mat(n)$,
\item $\overline{\cG_\mat(n)} = \cG(n)$, for all $n \ge 3$, if and only if CEP has a positive answer.
\end{enumerate}
\noindent
With the negative answer to CEP established in \cite{JNVWY:MIP*=RE}, we get  $\overline{\cG_\mat(n)} \subsetneq \cG(n)$, for some $n \ge 3$.
\end{remark}

\noindent
Consider also the sets $\cD_\mat(n) \subseteq \cD_\fin(n) \subseteq \cD_I(n) \subseteq \cD(n)$ of $n \times n$ matrices over $\R$, where we replace unitaries in the definition above with projections. The diagonal entries of any matrix in $\cD_\mat(n)$ are rational numbers. Hence $\cD_\mat(n)$ is neither closed nor convex, for $n \ge 2$. On the other hand, using direct sums, it is easy to see that the sets $\cG_\fin(n)$, $\cG_I(n)$, $\cG(n)$, $\cD_\fin(n)$, $\cD_I(n)$, and $\cD(n)$ are convex, for all $n \ge 1$. 

A more refined argument, using a theorem by Kruglyak, Rabanovich and Samoilenko, \cite{KRS:projections}, on existence of projections on a Hilbert space adding up to a scalar multiple of the identity, it was shown in \cite{MusRor:Infdim} that $\cD_\fin(n)$ is not closed, for $n \ge 5$; see also \cite{DykPauPra:non-closure}.

Using the trick in \cite[Proposition 3.2]{MusRor:Infdim} (arising from a discussion with W.\ Slofstra), there is a relation between the sets $\cD_*(n)$ and $\cG_*(2n+1)$, which leads to the following:

\begin{theorem}[Musat--R\o rdam] \label{thm:MR-notclosed} \mbox{}
\begin{enumerate}
\item $\cG_\mat(n)$ is not closed and not convex, for $n \ge 3$.
\item $\cG_\fin(n)$ and $\cG_I(n)$ are not closed, for $n \ge 11$.
\end{enumerate}
\end{theorem}

\noindent Part (i) is proved in  \cite[Corollary 3.3]{MusRor:Infdim} and part (ii) in 
\cite[Theorem 3.6]{MusRor:Infdim}. This implies in particular that sets $\cG_k(n) \subseteq \cG_\mat(n)$ do not stabilize for large $k$, for each $n \ge 3$; and also that one cannot exhaust $\cG_\fin(n)$ with a single finite-dimensional \Cs{} $\cA$. 

Using a clever determinant trick and writing a scalar multiple of the unit of a von Neumann algebra as a product of symmetries, Mori improved in \cite{Mori:shape}  the bound on $n$ as follows:

\begin{theorem}[Mori] \label{thm:Mori} $\cG_\fin(n)$ and $\cG_I(n)$ are not closed, for $n \ge 8$.
\end{theorem} 

\noindent The non-closure of $\cG_\fin(n)$, for $n \ge 8$, is precisely \cite[Theorem 2.3]{Mori:shape}, while the non-closure of $\cG_I(n)$ follows from the proof of that theorem. Indeed, for $n \ge 8$ and $\theta \in \R \setminus \Q$, Mori finds unitaries $u_1, \dots, u_8$ in $\cR$ such that if $(\cM,\tau)$ is a tracial von Neumann algebra with unitaries $v_1, \dots, v_8$ satisfying $[\tau_\cR(u_iu_j^*)] = [\tau(v_iv_j^*)]$, then $v_1^*v_2, v_2^*v_3, v_3^*v_4$ and $e^{-2\pi i\theta}v_1^*v_4$ are symmetries which satisfy
$$e^{2\pi i \theta} \,1_\cM = (v_1^*v_2)(v_2^*v_3) (v_3^*v_4)(e^{-2\pi i\theta}v_1^*v_4)^*.$$
If $\cM$ had a  representation on a finite-dimensional Hilbert space $H$, then $e^{2\pi i \theta} \,1_H$ would be a product of four symmetries, which is impossible (the determinant of a symmetry on a finite- dimensional Hilbert space is $\pm 1$). Hence $\cM$ has no finite-dimensional representations, and $\cM$ therefore is of type II$_1$. Thus $[\tau_\cR(u_iu_j^*)]$ does not belong to $\cG_I(n)$, but it does belong to the closure of $\cG_\mat(n)$ (by Theorem~\ref{thm:Kir1}) and hence to the closure of $\cG_I(n)$.

It follows from the above result of Dykema and Juschenko that $\cG_\fin(n)$ is closed, for $n \le 3$. It remains unknown if $\cG_\fin(n)$ is closed for $n=4,5,6,7$. 

\section{Factorizable quantum channels and their ancillas} \label{sec:factorizable}

\noindent Let $\tr_n$ and $\Tr_n$ denote the normalized and the standard trace, respectively, on the \Cs{} $M_n(\C)$ of complex $n \times n$ matrices, for $n \ge 2$.
A linear map $\Phi \colon M_n(\C) \to M_n(\C)$ is a \emph{quantum channel} if it is completely positive and trace preserving (i.e., $\tr_n \circ \Phi = \tr_n$). We denote by $\UCPT(n)$ the set of all unital quantum channels in dimension $n$. By Choi's celebrated theorem from 1973, \cite{Choi:CP}, a linear map 
$\Phi \colon M_n(\C) \to M_n(\C)$ is completely positive if and only if its Choi matrix $C_\Phi$ is positive, respectively, if and only if $\Phi(x) = \sum_{j=1}^d v_jxv_j^*$, for all $x \in M_n(\C)$, for some \emph{Kraus operators} $v_1, \dots, v_d \in M_n(\C)$. The Choi matrix is defined to be $C_\Phi = \sum_{i,j=1}^n \Phi(e_{ij}) \otimes e_{ij} \in M_n(\C) \otimes M_n(\C)$, where $\{e_{ij}\}$ are the standard matrix units for $M_n(\C)$. 
The Kraus operators can be  chosen to be linearly independent, in which case $d$ is equal to the \emph{Choi rank} of $\Phi$, which is defined to be the rank of $C_\Phi$.
Clearly, $\Phi$ is unital if and only if $\sum v_jv_j^* = 1_n$, and $\Phi$ is trace preserving if and only if $\sum  v_j^*v_j= 1_n$. 

Examples of unital quantum channels in dimension $n$ include automorphisms of $M_n(\C)$ (which all are inner), and the so-called Schur channels, which are Schur multipliers $\Sigma_b$ associated to a correlation matrix $b$ in $M_n(\C)$. The set $\UCPT(n)$  is closed and convex, so any convex combination of the above examples also provides a unital quantum channel. 

\begin{definition} \label{def:factorizable}
A linear map $\Phi \colon M_n(\C) \to M_n(\C)$ is called \emph{factorizable} if it is of the form
\begin{equation} \label{eq:1}
\Phi(x) = (\Id_n \otimes \tau)(u(x \otimes 1_\cA)u^*), \qquad x \in M_n(\C),
\end{equation}
where $(\cA,\tau)$ is a tracial von Neumann algebra, called \emph{ancilla}, and $u$ is a unitary in $M_n(\C) \otimes \cA$. 
\end{definition}

\noindent Factorizable maps have appeared in thermodynamics, where they are referred to as ``noisy operations'', see \cite{BZ:Geometry}. If $(\cA,\tau) = (M_k(\C),\tr_k)$, for some $k \ge 2$, then $\Phi$ is called \emph{$k$-noisy}. Note that factorizable maps always are unital quantum channels. The set of all factorizable maps in dimension $n$ is denoted by $\FM(n)$. 

By the original definition, due to Anantharaman-Delaroche, \cite{A-D:factorizable}, a linear map $\Phi \colon M_n(\C) \to M_n(\C)$ is factorizable if there exist a tracial von Neumann algebra $(\cA,\tau)$ and unital \sh s 
$\alpha, \beta \colon M_n(\C) \to M_n(\C) \otimes \cA$, so that $\Phi = \beta^* \circ \alpha$:
\begin{equation} \label{eq:A-D}
\begin{split}
\xymatrix@C-1.0pc{
 {{M}_n(\mathbb{C})}\ar@{->}^{\Phi}[rr]
 \ar@{->}_{\alpha}[dr]
 & & {M}_n(\mathbb{C}) \ar[dl]_\beta & & \vspace*{0.1cm}\\
& {{M}_n(\mathbb{C})\otimes \cA} 
\ar@/_/[ur]_-{\; \;  \beta^*}
}
\end{split}
\end{equation}
where  $\beta^*$ is the (Euclidean) adjoint of $\beta$, i.e., $\langle \beta(x), y \rangle_{\Tr_n \otimes \tau} = \langle x, \beta^*(y) \rangle_{\Tr_n}$, for all $x \in M_n(\C)$ and $y \in M_n(\C) \otimes \cA$. We can take $\alpha$ to be $\Id_n \otimes 1_\cA$. Using that $\alpha$ and $\beta$ automatically are unitarily equivalent (since $M_n(\C) \otimes \cA$ is a von Neumann algebra, see \cite[Lemma 2.1]{HaaMusat:CMP-2011}), it was shown in 
\cite[Theorem 2.2]{HaaMusat:CMP-2011} that the definition given in \eqref{eq:1} and the original definition above are equivalent in a way that preserves the ancilla $\cA$. In the situation of \eqref{eq:A-D} we sometimes say that $\Phi$ factors (or has a factorization) through $M_n(\C) \otimes \cA$.

One can explicitly recover $\Phi$ from $\alpha$ and $\beta$ via the formula:
\begin{equation} \label{eq:2}
\Phi(x) = \sum_{i,j=1}^n \langle \alpha(x), \beta(e_{ij}) \rangle_{\Tr_n \otimes \tau} \, e_{ij}, \qquad x \in M_n(\C).
\end{equation}

\begin{remark}[Stinespring representations] The expression for a factorizable map in \eqref{eq:1} resembles the Stinespring representation for completely positive maps  $\Phi \colon M_n(\C) \to M_n(\C)$. Indeed, by Stinespring's theorem,  $\Phi$ is completely positive and trace preserving (i.e., a quantum channel) if and only if there exist $r \ge 1$ and a unitary $u \in M_n(\C) \otimes M_r(\C)$ such that
$$\Phi(x) = (\Id_n \otimes \Tr_r)(u(x \otimes e_{11})u^*), \qquad x \in M_n(\C),$$
where again $e_{11}$ is the 1-dimensional projection with $1$ in the $(1,1)$th position;  and $\Phi$ is unital and completely positive  if and only if there exist $r \ge 1$, a unitary $u \in M_n(\C) \otimes M_r(\C)$  and a state $\rho$ (that can be taken to be pure) such that
$$\Phi(x) = (\Id_n \otimes \rho)(u(x \otimes 1_r)u^*), \qquad x \in M_n(\C).$$
In both cases one can take $r$ to be the Choi rank of $\Phi$ (which is at most $n^2$). 

The expression for factorizable maps in Definition~\ref{def:factorizable} unites the two expressions above. However, there are two remarkable differences: Not all UCPT maps are factorizable, and  there seems to be no obvious Stinespring type characterization of UCPT maps; and the ancilla $\cA$ for factorizable maps cannot be controlled. In fact, as we shall see, we will need infinite-dimensional, even non-type I, ancillas, for certain factorizable maps; and in presence of a negative answer to the CEP we will even need ancillas that do not embed into $\cR^\omega$!
\end{remark}

\begin{example} \mbox{}
\begin{enumerate}
\item All automorphisms of $M_n(\C)$ are factorizable. Automorphisms on $M_n(\C)$ are always inner, i.e., of the form $\Phi(x) = uxu^*$, for some unitary $u \in M_n(\C)$. One can therefore take the trivial ancilla $\cA = \C$ to express $\Phi$ as a factorizable map. 
\item A Schur channel $\Sigma_b$ on $M_n(\C)$, with $b \in M_n(\C)$, is factorizable if and only if $b \in \cG(n)$, in which case, if $b = [\tau(u_iu_j^*)]$, with $u_1, \dots, u_n$ unitaries in an ancillary tracial von Neumann algebra $(\cA,\tau)$, we can take $u = \sum_{j=1}^n e_{jj} \otimes u_j$ in $M_n(\C) \otimes \cA$, see 
\cite[Proposition 2.8]{HaaMusat:CMP-2011}.  In particular, $\Sigma_b$ is $k$-noisy if and only if $b \in \cG_k(n)$.
\end{enumerate}
\end{example}

\noindent
The set $\FM(n)$ of factorizable maps in dimension $n$ is closed (hence compact) and convex, and also closed under composition. To see the latter, if $\Phi$ and $\Psi$ are factorizable maps in dimension $n$ with representations $\Phi(x) = (\Id_n \otimes \tau_\cA)(u(x\otimes 1_\cA)u^*)$ and $\Psi(x) = (\Id_n \otimes \tau_\cB)(v(x\otimes 1_\cB)v^*)$, then 
$$(\Psi \circ \Phi)(x) = (\Id_n \otimes \tau_\cA \otimes \tau_\cB)(w(x \otimes 1_\cA \otimes 1_\cB)w^*), \quad x \in M_n(\C),$$
with $w = \widehat{v} \, (u \otimes 1_\cB)$, where $\widehat{x \otimes b}= x \otimes 1_\cA \otimes b$, for $x \in M_n(\C)$ and $b \in \cB$. In particular,
\begin{equation} \label{eq:chain}
\conv(\Aut(M_n(\C))) \subseteq \FM(n) \subseteq \UCPT(n), \qquad n \ge 2.
\end{equation}

Channels belonging to $\conv(\Aut(M_n(\C)))$ are often called \emph{mixtures of unitaries}. The inclusions in \eqref{eq:chain} above are equalities throughout when $n=2$, as shown by K\"ummerer, \cite{Kuem:2x2}. However, for $n \ge 3$,
$\conv(\Aut(M_n(\C))) \ne \UCPT(n)$, that is, the Quantum Birkhoff Conjecture (QBC) fails, for all $n \ge 3$. This was shown by K\"ummerer in \cite{Kuem:Hab}, for $n=3$, by K\"ummerer--Maassen, \cite{Kuem-Maassen}, for $n \ge 4$; see also Landau--Streater, \cite{Landau-Streater} for a simpler proof. 

We remark here that the Holevo--Werner channel $W_3^-$ in dimension $n=3$, belongs to $\UCPT(3)$ but not to 
$\conv(\Aut(M_n(\C)))$, thus witnessing that  $\conv(\Aut(M_3(\C))) \ne \UCPT(3)$. The ``negative'' Holevo-Werner channels are defined by $W_n^-(x) = (n-1)^{-1} (\Tr_n(x) 1_n - x^t)$, for $n \ge 2$, where $x^t$ is the transpose of $x \in M_n(\C)$. They are completely positive because their Choi matrix, which is  $(n-1)^{-1}(1_n \otimes 1_n -s_n)$, where $s_n \in M_n(\C) \otimes M_n(\C)$ is the ``flip-symmetry'', is positive. The Choi representation for $W_n^-$ is given by
$$W_n^-(x) =  \frac{1}{2n-2} \sum_{i,j=1}^n (e_{ij}-e_{ji})x (e_{ij}-e_{ji})^* = \frac{1}{n-1} \sum_{i < j}^n (e_{ij}-e_{ji})x (e_{ij}-e_{ji})^*.$$ 
In particular, $W_3^-(x) = v_1xv_1^*+v_2xv_2^*+v_3xv_3^*$, where
$$v_1 = \frac{1}{\sqrt{2}} \begin{pmatrix}0 & 1 & 0 \\ -1 & 0 & 0 \\ 0 & 0 & 0\end{pmatrix}, \quad 
v_1 = \frac{1}{\sqrt{2}} \begin{pmatrix}0 & 0 & 1 \\ 0 & 0 & 0 \\ -1 & 0 & 0\end{pmatrix}, \quad 
v_1 = \frac{1}{\sqrt{2}} \begin{pmatrix}0 & 0 & 0 \\ 0 & 0 & 1 \\ 0 & -1 & 0\end{pmatrix}.
$$ One can check that $\{v_iv_j^* : 1 \le i,j \le 3\}$ forms a linearly independent set; hence $W_3^-$ is an extreme point of $\UCPT(3)$, cf.\ \cite[Theorem 5]{Choi:CP}. If $W_3^-$ were to belong to 
$\conv(\Aut(M_3(\C)))$ it would be an extreme point here, hence an automorphism on $M_3(\C)$, which it clearly is not. 

It further follows from \cite[Corollary 2.3]{HaaMusat:CMP-2011} that $W_3^-$ is not factorizable, witnessing that the second inclusion in \eqref{eq:chain} is strict in general. Examples of non-factorizable channels in all other dimensions were further established in \cite{HaaMusat:CMP-2011}. It was also shown therein that each non-factorizable channel violates the \emph{Asymptotic Quantum Birkhoff Conjecture} formulated by Smolin--Verstraete--Winter, \cite{SmoVerWin:entagle}, as an attempted restoration of the QBC (in the asymptotic limit). Non-factorizable maps are directly  connected to the study of the convex structure of $\UCPT(n)$ and the related classes, $\mathrm{CPT}(n)$ and $\mathrm{UCP}(n)$, as every unital quantum channel which is an extreme point in either $\mathrm{CPT}(n)$ or $\mathrm{UCP}(n)$ is non-factorizable, cf.\ \cite[Corollary 2.3]{HaaMusat:CMP-2011}. 
For a systematic recipe for constructing non-factorizable channels, and more details of the convex structure of these sets, see also \cite{HMR:extreme}. 

In \cite[Example 3.3]{HaaMusat:CMP-2011}, an example is given of a $6 \times 6$ matrix $b$ whose associated Schur channel $\Sigma_b$ is factorizable, in fact $k$-noisy, for some $k$, \cite{Ricard:Markov}, but not a mixture of unitaries, so the first inclusion in \eqref{eq:chain} is also strict in general. (This provides another concrete example of the failure of QBC in dimension $n=6$.)

\begin{example}[Mixtures of unitaries] \label{ex:mixtures}
It was shown in \cite[Proposition 2.4]{HaaMusat:CMP-2011} that a factorizable channel $\Phi$ in dimension $n$ belongs to $\conv(\Aut(M_n(\C)))$, if and only it has an abelian ancilla. More specifically, if $\Phi(x) = \sum_{j=1}^N t_j u_jxu_j^*$, with $u_1, \dots, u_N$ unitaries in $M_n(\C)$, $t_1, \dots, t_N \ge 0$ and $\sum t_j = 1$, then $\Phi$ has an ancilla $(\cA,\tau)$ where $\cA=\C^N$ with trace $\tau$ that has weight $t_j$ on the $j$th copy of $\C$. With $u = (u_1, \dots, u_N) \in M_n(\C) \otimes \C^N$  we then get a representation for $\Phi$ as in \eqref{eq:2}. 
Note that $\cA$ has a trace preserving embedding into $M_k(\C)$ if and only if $t_1, \dots, t_N \in k^{-1} \Z$, in which case $\Phi$ is $k$-noisy. There are examples of mixtures of unitaries that are not $k$-noisy, for any $k \ge 2$, see Proposition~\ref{prop:k-noisy} below.
\end{example}

\begin{example}[The completely depolarizing channel] \label{ex:depolarizing}
In this example we shall see how the same channel can be represented by several different ancillas, also when assuming that the ancillas are ``minimal'', in the sense that $M_n(\C) \otimes \cA$ is generated as a von Neumann algebra by the images of $\alpha$ and $\beta$ in the situation of \eqref{eq:A-D}. 

We consider the \emph{completely depolarizing channel} $S_n$, for $n \ge 2$, defined by $S_n(x) = \tr_n(x) 1_n$, for $x \in M_n(\C)$. Its Choi matrix is $n^{-1} 1_n \otimes 1_n$, so it has Choi rank $n^2$. As
$$S_n(x) = \int_{\cU(n)} uxu^* \, du,$$
where $\cU(n)$ denotes the compact group of unitaries in $M_n(\C)$ and $du$ its Haar measure, we see that $S_n$ belongs to $\conv(\Aut(M_n(\C))$ (which is closed by Caratheodory's theorem). We can also see this more directly. Let
$$u = \begin{pmatrix} 1 & 0 & \cdots & 0 \\ 0 & \omega & \cdots & 0 \\ \vdots & & \ddots & 0 \\ 0 & 0 & \cdots & \omega^{n-1} \end{pmatrix}, \qquad v = \begin{pmatrix} 0 & 1 & 0 & \cdots & 0 \\ 0 & 0 & 1 & \cdots & 0 \\ \vdots & \vdots & & \ddots & 0 \\ 0 & 0 & 0 & \cdots & 1 \\ 1 & 0 & 0 & \cdots & 0\end{pmatrix},
 $$
be the Voiculescu $n \times n$ unitaries, where $\omega = \exp(2\pi i/n)$. Then
$$S_n(x) = \frac{1}{n^2} \sum_{i,j=1}^n v^iu^j x u^{-j}v^{-i}, \qquad x \in M_n(\C),$$
which explicitly belongs to $\conv(\Aut(M_n(\C))$. Any expression of $S_n$ as a convex combination in $\conv(\Aut(M_n(\C))$  must have at least $n^2$ terms, as the Choi rank of $S_n$ is $n^2$, so the expression above is the most economical one. It follows from \cite[Proposition 2.4]{HaaMusat:CMP-2011} that $S_n$ admits a factorization over $M_n(\C) \otimes \C^{n^2}$, i.e., with ancilla $\C^{n^2}$ equipped with the trace that has weight $1/n^2$ on each copy of $\C$, cf.\ Example~\ref{ex:mixtures}.

Here are two other ways of expressing $S_n$ as factorizations through tracial von Neumann algebras:
$$
\xymatrix@C-2.0pc{M_n(\C) \ar[dr]_-\alpha \ar[rr]^-{S_n}&& M_n(\C) \ar[dl]^-\beta \\ & M_n(\C) \otimes M_n(\C) &} 
\qquad 
\xymatrix@C-3.0pc{M_n(\C) \ar[dr]_-{\iota_1} \ar[rr]^-{S_n}&& M_n(\C) \ar[dl]^-{\iota_2} \\ & (M_n(\C), \tr_n) * (M_n(\C), \tr_n) &} 
$$
where $\alpha = \Id_n \otimes 1_n$ and $\beta = 1_n \otimes \Id_n$; let further
$$(\cB,\tau) := (M_n(\C), \tr_n) * (M_n(\C), \tr_n)$$ denote the reduced free product von Neumann algebra with its canonical inclusions $\iota_1$ and $\iota_2$ of $M_n(\C)$. We can use \eqref{eq:2}  together with the formula
$$\tau_\cB(\iota_1(x)\iota_2(y)) = (\tr_n \otimes \tr_n)(x \otimes y) = \tr_n(x) \tr_n(y), \qquad x, y \in M_n(\C),$$
to verify that these two factorizations represent $S_n$. The corresponding ancillas are $M_n(\C)$ and  the corner $\iota_1(e_{11}) \, \cB \, \iota_1(e_{11})$ of the (non-hyperfinite) II$_1$-factor $\cB$, respectively.
\end{example}

\noindent The question whether all factorizable channels are $k$-noisy, for some $k \ge 2$, arose naturally. In dimension $2$, this was settled in the positive by M\"uller-Hermes and Perry, \cite{Muller-Hermes-Perry:4-noisy}. More precisely, they showed that every qubit (= quantum channel in dimension $2$) is $4$-noisy. By contrast, in every dimension $n \ge 3$, there exist factorizable channels (even mixtures of unitaries) which are not $k$-noisy for any $k \ge 2$. This result was announced during the thematic program on operator algebras and  QIT at IHP, Paris, in the Fall of 2017, and was referenced in \cite{MusRor:Infdim}. We take this opportunity to include here the (thus far unpublished) details of the proof. 

\begin{proposition} \label{prop:k-noisy}
Let $n \ge 3$. Given unitaries $u_1,u_2$ in $M_n(\C)$ and $0 < t < 1$, let $\Phi(x) = t \, u_1xu_1^* + (1-t) \, u_2xu_2^*$, for $x \in M_n(\C)$. Assume that the set $\{1_n, u_1^*u_2, u_2^*u_1\}$ is linearly independent. Then $\Phi$ is $k$-noisy, for some $k \ge 2$ if and only if $t \in k^{-1} \Z$.
\end{proposition}

\noindent Unitaries $u_1, u_2$ in $M_n(\C)$ such that the set $\{1_n, u_1^*u_2, u_2^*u_1\}$ is linearly independent exist if and only if $n \ge 3$. Indeed, let $\omega \in \C \setminus \{\pm 1\}$, with $|\omega|=1$, and note that $\{(1,1,1), (1,\omega, \bar{\omega}), (1, \bar{\omega}, \omega)\}$ is a linearly independent set in $\C^3$. For $n \ge 3$, we can therefore take $u_1 = 1_n,$ and $u_2$ and $u_3$ to be diagonal unitaries in $M_n(\C)$ whose first three diagonal entries are $1,\omega, \bar{\omega}$ and r$1, \bar{\omega}, \omega$, respectively. Conversely, suppose that $\{1_n, u_1^*u_2, u_2^*u_1\}$ is linearly independent, and set $v = u_1^*u_2$. Then $\{1_n,v,v^*\}$ is linearly independent. Diagonalizing $v$, we see that this is possible only if $n \ge 3$.

\begin{proof} Set $v_1 = \sqrt{t} \, u_1$ and $v_2 = \sqrt{1-t} \, u_2$. Then $\Phi(x) = v_1xv_1^*+v_2xv_2^*$ is a Choi representation of $\Phi$ with $\{v_1,v_2\}$ linearly independent. An application of \cite[Theorem 2.2]{HaaMusat:CMP-2011} (and its proof) shows that $\Phi$ is $k$-noisy, for some $k \ge 2$, if and only if there exist $w_1,w_2 \in M_k(\C)$ such that $u:= v_1 \otimes w_1 + v_2 \otimes w_2$ is a unitary in $M_n(\C) \otimes M_k(\C)$ and $\tr_k(w_jw_i^*) = \delta_{ij}$, for $1 \le i,j \le 2$. 

If $t \in k^{-1} \Z$, then there exists a projection $p$ in $M_k(\C)$ with $\tr_k(p) = t$, which makes $u:= u_1 \otimes p + u_2 \otimes (1_k-p)$ a unitary in $M_n(\C) \otimes M_k(\C)$. Hence the conditions of \cite[Theorem 2.2]{HaaMusat:CMP-2011}  hold with $w_1 = t^{-1/2}p$ and $w_2 = (1-t)^{-1/2}(1_k-p)$, attesting that $\Phi$ is $k$-noisy.

For the reverse implication, assume $\Phi$ is $k$-noisy, for some $k \ge 2$, and let $w_1,w_2 \in M_k(\C)$ be as above. Set $z_1 = \sqrt{t} \, w_1$ and $z_2 = \sqrt{1-t} \, w_2$, so that $u:= u_1 \otimes z_1 + u_2 \otimes z_2$ is a unitary in $M_n(\C) \otimes M_k(\C)$, and $\tr_k(z_j^*z_i)=0$ when $i \ne j$ and $\tr_k(z_1^*z_1)= t$ and $\tr_k(z_2^*z_2)= 1-t$. We deduce that
$$1_n \otimes 1_k = u^*u = 1_n \otimes (z_1^*z_1 +z_2^*z_2) + u_2^*u_1 \otimes z_2^*z_1 + u_1^*u_2 \otimes z_2^*z_2.$$
By linear independence of $\{1_n, u_2^*u_1, u_1^*u_2\}$ we conclude that 
\begin{equation} \label{eq:3}
1_k = z_1^*z_1 +z_2^*z_2, \qquad 0 = z_2^*z_1 = z_1^*z_2.
\end{equation}
Let $p_j \in M_k(\C)$ be the range projections of $z_jz_j^*$, for $j=1,2$. Then $p_1 \perp p_2$ by \eqref{eq:3}. By \eqref{eq:3} we also get that $\|z_j\| \le 1$, so $z_jz_j^* \le p_j$, for $j=1,2$. Further,
$$1 = t+(1-t) = \tr_k(z_1 z_1^*) + \tr_k(z_2 z_2^*) \le \tr_k(p_1) + \tr_k(p_2) \le \tr_k(1_k) = 1.$$
Thus both inequalities above are equalities; in particular, we have $\tr_k(p_1) = \tr_k(z_1z_1^*) = t$, which implies that $t \in k^{-1} \Z$, as wanted.
\end{proof}

\noindent By the same proof as above one can arrive at the slightly stronger statement, namely that $\Phi$ in the proposition admits a factorization with ancilla $(\cA,\tau)$ if and only if $\cA$ contains a projection $p$ with $\tau(p)=t$.

In view of  Proposition~\ref{prop:k-noisy}, a natural question arises: Can a factorizable channel $\Phi$ always be approximated by $k$-noisy channels, where $k$ is allowed to vary? 

\begin{theorem}[Haagerup--Musat, {\cite[Theorem 3.6]{HaaMusat:CMP-2015}}] \label{thm:HM-1}
Let $\Phi \in \FM(n)$ be given. Then there exists a sequence $\{\Phi_m\}$ of $k$-noisy channels in $\FM(n)$ (with $k$ depending on $m$) such that $\lim_{m \to \infty} \|\Phi-\Phi_m\|_{\mathrm{cb}} = 0$ if and only if $\Phi$ has a factorization as in Definition~\ref{def:factorizable}, or equivalently, as in \eqref{eq:A-D}, with an ancilla $\cA$ that embeds into $\cR^\omega$
\end{theorem}

\noindent As all norms on the finite-dimensional vector space of linear maps $M_n(\C) \to M_n(\C)$ are equivalent,  we can in the theorem above replace the cb-norm with any other norm. 

Let $\FM_\mat(n) \subseteq \FM_\fin(n) \subseteq \FM_I(n) \subseteq \FM(n)$ denote the set of factorizable maps in dimension $n$ that admit an ancilla that is a matrix algebra, finite-dimensional, and of type I, respectively. In this notation, $\FM_\mat(n)$ is the class of all $k$-noisy channels. The results below were mentioned in \cite{MusRor:Infdim}, with the promise of a proof in a forthcoming paper. We give below the announced proof.

\begin{proposition} \mbox{} \label{prop:FM}
\begin{enumerate}
\item $\FM_\mat(n)$ is not closed and not convex, for $n \ge 3$,
\item $\overline{\FM_\mat(n)} = \overline{\FM_\fin(n)}$ and $\conv(\FM_\mat(n)) = \FM_\fin(n)$.
\end{enumerate}
\end{proposition}

\begin{proof} (i) follows from Proposition~\ref{prop:k-noisy}. Indeed, $\Aut(M_n(\C)) \subseteq \FM_\mat(n)$, but $\conv(\Aut(M_n(\C)))$ is not contained in $\FM_\mat(n)$, when $n \ge 3$. Also, the channels considered in Proposition~\ref{prop:k-noisy} belong to $\FM_\mat(n)$ if and only if $t \in \Q \cap (0,1)$ which shows that  $\FM_\mat(n)$ is not closed.

(ii). For the first part we must show that $\FM_\fin(n) \subseteq \overline{\FM_\mat(n)}$. We can write a typical element of $\FM_\fin(n)$ as in \eqref{eq:1}:
$$\Phi_\tau(x) = (\Id_n \otimes \tau)(u (x \otimes 1_\cA)u^*), \qquad x \in M_n(\C),$$
for some finite-dimensional tracial von Neumann algebra $(\cA,\tau)$ and some unitary $u \in M_n(\C) \otimes \cA$. We fix $\cA$ and $u$ and let $\tau$ vary. If $\tau$ is ``rational'', i.e., $\tau$ takes rational values on all projections in $\cA$ (or just the minimal central ones), then $\cA$ has a trace preserving embedding into a  full matrix algebra $M_k(\C)$, in which case $\Phi_\tau$ is $k$-noisy. The set of rational traces is norm dense in the set of all traces in $\cA$, and the map $\tau \mapsto \Phi_\tau$ is continuous. This proves that each channel in $\FM_\fin(n)$ can be approximated by $k$-noisy ones. 

We proceed to prove the second part of the statement. Let $\Phi_1, \dots, \Phi_N \in \FM_\mat(n)$ and $t_1, \dots, t_N \ge 0$ with $\sum t_j = 1$ be given. Then, for $1 \le j \le N$,
$$\Phi_j(x) = (\Id_n \otimes \tr_{k_j})(u_j(x \otimes 1_{k_j})u_j^*), \qquad x \in M_n(\C),$$ 
for some $k_j \ge 2$ and some unitary $u_j$ in $M_n(\C) \otimes M_{k_j}(\C)$. Set $\cA = \bigoplus_{j=1}^N M_{k_j}(\C)$ and let $\tau$ be the trace on $\cA$ given by $\tau(y_1, \dots, y_N) = \sum t_j \, \tr_{k_j}(y_j)$, for $y_j $ in $M_{k_j}(\C)$. Identifying $M_n(\C) \otimes \cA=\bigoplus_{j=1}^N M_n(\C) \otimes M_{k_j}(\C)$, set $u = (u_1, \dots, u_N)$ in $M_n(\C) \otimes \cA$. Then 
$$\Big(\sum_{j=1}^N t_j \, \Phi_j\Big)(x) = (\Id_n \otimes \tau)(u (x \otimes 1_\cA)u^*), \qquad x \in M_n(\C),$$
which shows that $\sum_{j=1}^N t_j \, \Phi_j$ belongs to $\FM_\fin(n)$. 

Conversely, let $\Phi \in \FM_\fin(n)$, let $(\cA,\tau)$ be an ancilla for $\Phi$ and let $u \in M_n(\C) \otimes \cA$ be a unitary such that $\Phi$ has a presentation as in \eqref{eq:1}. Write $\cA = \bigoplus_{j=1}^N M_{k_j}(\C)$, let $q_j \in \cA$ be the unit for the $j$th summand, and set $t_j = \tau(q_j)$. We may assume that $t_j >0$, for all $j$ (if not, then we discard some of the summands in $\cA$). Then $q_j \cA q_j = M_{k_j}(\C)$ and $\sum q_j = 1_\cA$. Set
$$\Phi_j(x) = t_j^{-1} \, (\Id_n \otimes \tau)(u(x \otimes q_j)u^*), \qquad x \in M_n(\C).$$
Then $\Phi_j$ is $k_j$-noisy and $\Phi = \sum_{j=1}^N t_j \, \Phi_j$.
\end{proof}

\noindent
The result below is essentially contained in \cite[Proposition 2.8]{HaaMusat:CMP-2011} and 
\cite[Theorem 3.7]{HaaMusat:CMP-2015}, and their proofs. 

\begin{proposition} \label{prop:Schur}
 Let $b$ be a correlation matrix  in $M_n(\C)$, and consider the associated Schur channel $\Sigma_b$. Then $\Sigma_b$ is  factorizable with ancilla $(\cA, \tau)$ if and only if there are unitaries $u_1, \dots, u_n$ in $M_n(\C) \otimes \cA$ so that $b = [(\tr_n \otimes \tau)(u_iu_j^*)]$. In particular, 
\begin{enumerate}
\item $\Sigma_b$ belongs to $\FM(n)$ if and only if $b \in \cG(n)$.
\item $\Sigma_b$ belongs to $\FM_\mat(n)$, $\FM_\fin(n)$, and $\FM_I(n)$, respectively, if and only if $b$ belongs to $\cG_\mat(n)$, $\cG_\fin(n)$, and $\cG_I(n)$, respectively.
\item $\Sigma_b$ belongs to $\overline{\FM_\mat(n)}$, i.e., it is a limit of $k$-noisy channels (with $k$ allowed to vary) if and only if $b \in  \overline{\cG_\mat(n)}$.
\end{enumerate}
\end{proposition}

\begin{proof} The first part is \cite[Proposition 2.8]{HaaMusat:CMP-2011}; (i) and (ii) are immediate consequences thereof.

(iii). If $\{b_m\}$ is a sequence in $\cG_\mat(n)$ converging to $b$, then $\Sigma_{b_m} \to \Sigma_b$, and each $\Sigma_{b_m}$ belongs to $\FM_\mat(n)$. Conversely, suppose that $b \in \cG(n)$ does not belong to $\overline{\cG_\mat(n)}$. Let $(\cA,\tau)$ be an ancilla for $\Sigma_b$. Then $b = [(\tr_n \otimes \tau)(u_iu_j^*)]$, for some unitaries $u_1, \dots, u_n$ in $M_n(\C) \otimes \cA$. By Kirchberg's theorem (Theorem~\ref{thm:Kir1} above) and by the assumption on $b$, we conclude that $M_n(\C) \otimes \cA$, and therefore also $\cA$, do not embed into $\cR^\omega$. By Theorem~\ref{thm:HM-1}, this implies that $\Sigma_b$ is not a limit of $k$-noisy channels.
\end{proof}

\noindent Combining Proposition~\ref{prop:Schur} with Theorem~\ref{thm:MR-notclosed} it was shown in \cite{MusRor:Infdim} that there are factorizable channels in all dimensions $n \ge 11$ that do not factor through a tracial von Neumann algebra of type I, i.e., require an ancilla of type II$_1$. With Theorem~\ref{thm:Mori} of Mori and the added comment below it, we extend the above mentioned results to all $n \ge 8$. 

\begin{corollary}[Musat--R\o rdam, Mori] For each $n \ge 8$ we have $\FM_I(n) \subsetneq \FM(n)$, i.e., there are factorizable channels in all these dimensions that require an ancilla of type II$_1$.
\end{corollary}

\noindent With the negative answer to CEP obtained in \cite{JNVWY:MIP*=RE}, we get $ \overline{\cG_\mat(n)} \subsetneq \cG(n)$, for some $n \ge 4$ (as remarked in (vi) below Theorem~\ref{thm:Kir1}), which has the following remarkable consequence for the existence of exotic factorizable maps:

\begin{corollary} Let $b \in \cG(n) \setminus  \overline{\cG_\mat(n)}$, for some $n \ge 4$. 
Then $\Sigma_b$ is factorizable but  not a limit of $k$-noisy channels, and no ancilla for $\Sigma_b$ admits an embedding into $\cR^\omega$. 
\end{corollary}

\section{The Connes-Kirchberg problem and Tsirelson's conjecture} \label{sec:Tsirelson}

\noindent Tsirelson's Conjecture concerns (equality of) sets of quantum correlations arising in two different quantum models of non-locality: the spatial (tensor product) and commutativity of observables, respectively. We shall now describe in more detail  these sets, along with the classical one.

A quantum system is emitted from a source and two scientists, Alice and Bob, residing in spatially separated labs, each receives part of the system and performs measurements on it. The state of the system is described by a wave function (a unit vector in a Hilbert space), observables are self-adjoint operators on the Hilbert space, and one measures corresponding eigenvalues of the observable. Say that Alice and Bob measure any one of $n$ possible observables, each with $k$ possible outcomes. The collection of probabilities $p(a,b \, | \, x,y)$, that Alice gets outcome $a$ when measuring observable $x$ and Bob gets outcome $b$ when measuring observable $y$, gives rise to an $nk \times nk$ real matrix, $[p(a,b \, | \, x,y)]_{a,b;x,y}$, a matrix of correlations. Mathematical models for classical and quantum correlations and the related concept of entanglement are at the heart of the Einstein--Podolsky--Rosen (EPR) paradox from 1935.

According to the classical model, there exist a probability space $(\Omega,\mu)$ (of hidden variables) and measurable partitions $\{X_a^x\}_{1 \le a \le k}$ and $\{Y_b^y\}_{1 \le b \le k}$ (one such partition for each pair of observables $(x,y)$) such that $p(a,b \, | \,  x,y) = \mu(X_a^x \cap Y_b^y)$, $1 \le a,b \le k$ and $1 \le x,y \le n$. One denotes by $C_c(n,k)$ the set of $nk \times nk$ correlation matrices as above, where the probability space and the corresponding measurable partitions are allowed to vary. It can be shown that $C_c(n,k)$ is convex and closed. It is a polytope and its extreme points can be described.

In the quantum case, for each observable that Alice and Bob can measure, they each have a projection-valued measure (PVM) at their disposal. Recall that a PVM on a Hilbert space $H$ is a $k$-tuple $P_1, \dots, P_k$ of projections with sum $I_H$.

There are two quantum models for interpreting the physical separation of the labs:
\begin{itemize}
\item \emph{Tensor product} model, according to which there exist Hilbert spaces $H_A$ and $H_B$, PVMs $\{P_a^x\}_{1 \le a \le k}$ on $H_A$ and $\{Q_b^y\}_{1 \le b \le k}$ on $H_B$, for each $(x,y)$, and a unit vector $\psi \in H_A \otimes H_B$ such that
$$p(a,b \, | \, x,y) = \langle (P_a^x \otimes Q_b^y) \psi, \psi \rangle, \quad 1 \le a,b \le k, \; 1 \le x,y \le n.$$
\item \emph{Commutativity of observables} model, according to which there exist a single Hilbert space $H$, commuting PVMs 
$\{P_a^x\}_{1 \le a \le k}$  and $\{Q_b^y\}_{1 \le b \le k}$ on $H$, and a unit vector $\psi \in H$ such that
$$p(a,b \, | \, x,y) = \langle (P_a^x  Q_b^y) \psi, \psi \rangle,   \hspace{.65cm} 1 \le a,b \le k, \; 1 \le x,y \le n.$$
\end{itemize}
The corresponding sets of correlations are denoted by $C_{qs}(n,k)$ (in the tensor product formalism) and $C_{qc}(n,k)$ (in the commuting formalism). It can be shown, using a direct sum argument, that both sets are convex. We summarize the relations between these sets of correlations in the diagram below, where $C_{qs}^{\mathrm{fin}}(n,k)$ and $C_{qc}^{\mathrm{fin}}(n,k)$ are as above with the restriction, that the Hilbert spaces $H_A$, $H_B$, and $H$ are required to be finite-dimensional; and  $C_{qa}(n,k)$ denotes the closure of $C_{qs}(n,k)$. Customarily, $C_{qs}^{\mathrm{fin}}$ is also denoted by $C_q$. 
\begin{equation} \label{eq:correlations}
\begin{split}
\xymatrix@C-.4pc@R-.6pc{ & C_{qs}^{\mathrm{fin}}(n,k) \ar@{}[d]|-*[@]{\subseteq}&\overset{\text{Tsirelson}}{=}& C_{qc}^{\mathrm{fin}}(n,k) \ar@{}[d]|-*[@]{\subseteq}& \\
C_c(n, k)  \ar@{}[r]|-*[@]{\subseteq}  & C_{qs}(n, k)  \ar@{}[r]|-*[@]{\subseteq} &  C_{qa}(n, k) \ar@{}[r]|-*[@]{\subseteq}  & C_{qc}(n, k)
\ar@{}[r]|-*[@]{\subseteq}  &  M_{nk}([0,1]).}
\end{split}
\end{equation}

\noindent
The separation between classical and quantum correlations is witnessed by Bell's inequalities, \cite{Bell-1964},
which were verified experimentally to be violated starting early 1980s by Aspect and his team, with more and more sophisticated experiments being performed nowadays at labs around the world. Aspect, Clauser and Zeilinger were awarded the Nobel Prize in Physics in 2022 ``for experiments with entangled photons, establishing the violation of Bell inequalities and pioneering quantum information science''. 
Interestingly, the failure of the EPR ``hidden variable'' theory is also reflected in the fact that the universal constant in Grothendieck's inequality (in the real case) is strictly larger than 1, see \cite{Pisier:Grothendieck}. 

We now state the theorem of Fritz, \cite{Fritz:Tsirelson}, independently obtained by Junge and co-authors, \cite{JNPP-GSW:Tsirelson}, describing the correlation sets $C_{qa}(n,k)$ and $C_{qc}(n,k)$ in terms of states on tensor products of certain naturally associated group \Cs s. 
Let $k,n \ge 2$ be given and let $\Gamma_{k,n}$ be the free product of $n$ copies of the cyclic group $\Z/k$ of order $k$. It follows that its full
group \Cs{} $C^*(\Gamma_{k,n})$ is the unital ($C^*$-)free product of $n$ copies of $C^*(\Z/k) = \C^k$. Let $e_a^x \in C^*(\Gamma_{k,n})$ 
denote the $a$th minimal projection in the $x$th copy of $C^*(\Z/k)$, which then is equal to
$\mathrm{span}(\{e_a^x : 1 \le a \le k\})$. 

\noindent By the universal property of $\Gamma_{k,n}$ and the max tensor product, if $\{P_a^x\}_{1 \le a \le k}$ and $\{Q_b^y\}_{1 \le b \le k}$ are commuting PVMs on a Hilbert space $H$, for fixed $x,y$, there are representations $\varphi \colon C^*(\Gamma_{k,n}) \to B(H)$, $\psi \colon C^*(\Gamma_{k,n}) \otimes_\mathrm{max} C^*(\Gamma_{k,n}) \to B(H)$ such that
$$\varphi(e_a^x) = P_a^x, \qquad \psi(e_a^x \otimes e_b^y) = P_a^xQ_b^y, \qquad 1 \le a,b \le k.$$

For a unital \Cs{} $\cA$, let $S(\cA)$ denote the convex compact set of all its states.

\begin{theorem}[Fritz, Junge et al] \label{thm:Fritz}
Let $n, k \ge 2$ be given. Then
\begin{enumerate}
\item $C_{qa}(n,k) = \Big\{ \big[\rho(e_a^x \otimes e_b^y)\big]_{a,b;x,y} : \rho \in S(C^*(\Gamma_{k,n}) \otimes_{\mathrm{min}} C^*(\Gamma_{k,n}))\Big\}$,
\item $C_{qc}(n,k) = \Big\{\big[\rho(e_a^x \otimes e_b^y)\big]_{a,b;x,y} : \rho \in S(C^*(\Gamma_{k,n}) \otimes_\mathrm{max} C^*(\Gamma_{k,n}))\Big\}$.
\end{enumerate}
\end{theorem}

\noindent 
For completeness of the exposition, we include the quite short and interesting proofs, using $C^*$-techniques, of Tsirelson's theorem and the density of the inclusion $C^{\mathrm{fin}}_{qs}(n,k) \subseteq C_{qs}(n,k)$, cf.\ \eqref{eq:correlations}.
The \Cs{} $C^*(\Gamma_{k,n})$ is residually finite-dimensional (RFD), since it is the unital free product of $n$ copies of the finite-dimensional \Cs{} $C^*(\Z/k)$, cf.\ \cite{Exel-Loring:RFD}.

\begin{proposition} \label{prop:dense}
$C^{\mathrm{fin}}_{qs}(n,k)$ is dense in $C_{qs}(n,k)$, for all $n, k\geq 2$.
\end{proposition}

\begin{proof} As $C^{\mathrm{fin}}_{qs}(n,k)$ by definition is a subset of $C_{qs}(n,k)$, it suffices to show that it is dense in $C_{qa}(n,k)$. We use Theorem~\ref{thm:Fritz} and the RFD property of $C^*(\Gamma_{k,n})$ to prove this. 

Let $\{\pi_r\}_{r \ge 1}$ be a separating family of representations of $C^*(\Gamma_{k,n})$ on finite- dimensional Hilbert spaces $H_r$. Set $H = \bigoplus_{r \ge 1} H_r$ and obtain a faithful representation $\pi = \bigoplus \pi_r$ of $C^*(\Gamma_{k,n})$ on $H$, which in turn gives a faithful representation $\pi \otimes \pi$ of $C^*(\Gamma_{k,n}) \otimes_{\mathrm{min}} C^*(\Gamma_{k,n})$ on $H \otimes H$. 
Identify $H^{(N)} = \bigoplus_{r=1}^N H_n$ with a subspace of $H$, and note that $H_0 :=\bigcup_{N \ge 1} H^{(N)} $ is dense in $H$. Let $\pi^{(N)} = \bigoplus_{r=1}^N \pi_r$. Consider the following finite-dimensional operator subspace of $C^*(\Gamma_{k,n}) \otimes_{\mathrm{min}} C^*(\Gamma_{k,n})$:
$$\cM = \mathrm{span}\{e_a^x \otimes e_b^y :1 \le a,b \le k, \, 1 \le x,y \le n\} $$
By restricting states on $C^*(\Gamma_{k,n}) \otimes_{\mathrm{min}} C^*(\Gamma_{k,n})$ to states on $\cM$, it follows from Theorem~\ref{thm:Fritz}  that $C_{qa}(n,k) = \{ \big[\rho(e_a^x \otimes e_b^y)\big]_{a,b;x,y} : \rho \in S(\cM)\}$. For  $\xi \in H \otimes H$, let $\omega_\xi$ be the corresponding vector state on $B(H \otimes H)$, and consider the state $\rho_\xi = \omega_\xi \circ (\pi \otimes \pi)$ on  $C^*(\Gamma_{k,n}) \otimes_{\mathrm{min}} C^*(\Gamma_{k,n})$. By \cite[Corollary 4.3.10]{KadRin:VolI}, 
$$S(\cM) = \overline{\conv} \{ \rho_\xi|_\cM : \xi \in H \otimes H, \, \|\xi\| =1\}.$$
Hence 
$$C_{qa}(n,k) = \overline{\conv}\{\big[\rho_\xi(e_a^x \otimes e_b^y)\big]_{a,b;x,y} : \xi \in H \otimes H\big\}.$$
If $\xi \in H^{(N)} \otimes H^{(N)}$, then
$\rho_\xi = \omega_\xi \circ (\pi^{(N)} \otimes \pi^{(N)})$, from which it follows that
$\big[\rho_\xi(e_a^x \otimes e_b^y)\big]_{a,b;x,y}$ belongs to $C^{\mathrm{fin}}_{qs}(n,k)$. Finally, as $H_0 \otimes H_0$ is dense in $H \otimes H$ and $C^{\mathrm{fin}}_{qs}(n,k)$ is convex, we obtain the claim of the proposition. 
\end{proof}

\begin{theorem}[Tsirelson] \label{thm:Tsirelson}
 $C^{\mathrm{fin}}_{qs}(n,k) = C^{\mathrm{fin}}_{qc}(n,k)$, for all $n,k \ge 2$.
\end{theorem}

\begin{proof} It is clear that $C^{\mathrm{fin}}_{qs}(n,k) \subseteq C^{\mathrm{fin}}_{qc}(n,k)$. To prove the reverse inclusion, take $T \in C^{\mathrm{fin}}_{qc}(n,k)$, and write $T =[ \langle P_a^x Q_b^y \xi, \xi \rangle]$, for some commuting PVMs $\{P_a^x\}$ and $\{Q_b^y\}$ on a finite dimensional Hilbert space $H$ and some unit vector $\xi \in H$. Let $\cA_1, \cA_2, \cA$ be the sub-\Cs s of $B(H)$ generated by 
$$\mathcal{P} := \{P_a^x : 1 \le x \le n, 1 \le a \le k\}, \quad \mathcal{Q} =  \{Q_b^y : 1 \le y \le n, 1 \le b \le k\}, \quad \mathcal{P} \cup \mathcal{Q},$$
respectively.  As the \Cs s $\cA_1, \cA_2$ commute and are nuclear (being finite-dimensional) we have a unital surjective $^*$-homomorphism $\psi \colon \cA_1 \otimes \cA_2 \to \cA$, satisfying $\psi(a \otimes b) = ab$, for $a \in \cA_1$ and $b \in \cA_2$. 

Consider the state $\rho = \omega_\xi \circ \psi$ on $\cA_1 \otimes \cA_2$, where $\omega_\xi$ is the vector state associated with the unit vector $\xi$, and note that $T = [\rho(P_a^x \otimes Q_b^y)]$. 
Since $\cA_1 \otimes \cA_2$ is a finite-dimensional \Cs{} acting faithfully on the Hilbert space $H \otimes H$, each state on $\cA_1 \otimes \cA_2$ is a convex combination of vector states. Hence we can write $\rho$ as a convex combination $\sum_{j =1}^N t_j \, \omega_{\xi_j}$ of vector states $\omega_{\xi_j}$ corresponding to unit vectors $\xi_j \in H \otimes H$. Thus
$$T = [\rho(P_a^x \otimes Q_b^y)] = \sum_{j=1}^N t_j \, [\langle (P_a^x \otimes Q_b^y) \xi_j, \xi_j \rangle] \in C_{qs}^\mathrm{fin}(n,k),$$
by convexity of the set $C_{qs}^\mathrm{fin}(n,k)$. 
\end{proof}

\noindent The extension of this result to the case when infinite-dimensional Hilbert spaces are considered, is known as the (strong) Tsirelson conjecture:

\begin{conjecture}[Tsirelson] One has $C_{qa}(n,k) = C_{qc}(n,k)$, for all $n,k \ge 2$; or, equivalently,  $C^{\mathrm{fin}}_{qc}(n,k)$ is dense in $C_{qc}(n,k)$, for all $n,k \ge 2$.
\end{conjecture}

\noindent It is an immediate consequence of Proposition~\ref{prop:dense} and Theorem~\ref{thm:Tsirelson} that the two formulations of Tsirelson's conjecture are equivalent, cf.\ \eqref{eq:correlations}. The first formulation is about interchangeability of the tensor product model and the commuting model for two separated physical systems, while the second formulation asks if the commuting model  in infinite dimensions can be approximated by the commuting model in finite dimensions.

In Kirchberg's seminal paper, \cite{Kir:CEP}, already mentioned in Section~\ref{sec:CEP}, CEP is recast in several ways. The key objects in Kirchberg's paper are the two \Cs s: $B(H)$, into which every separable \Cs{} embeds, and the universal \Cs{} $C^*(\mathbb{F}_\infty)$ of the free group, which surjects onto every separable \Cs; and two properties of \Cs s: the (local) lifting property (L)LP and Lance's weak expectation property WEP. 

Kirchberg conjectured that every \Cs{} is QWEP, i.e., is a quotient of a \Cs{} with WEP, and he proved that this conjecture is equivalent to CEP having a positive answer. He further proves that $B(H)$ has WEP, that $C^*(\mathbb{F}_\infty)$ is LP (hence LLP), and that for every \Cs{} $\cA$, the following statements hold:
\begin{eqnarray*}
\cA \; \text{has LLP}  \; & \iff & \cA \otimes_{\mathrm{min}} B(H) = \cA \otimes_{\mathrm{max}} B(H), \\
\cA \;  \text{has WEP}  &\iff & \cA \otimes_{\mathrm{min}}  C^*(\mathbb{F}_\infty)= \cA \otimes_{\mathrm{max}} C^*(\mathbb{F}_\infty).
\end{eqnarray*}
In particular, $C^*(\mathbb{F}_\infty) \otimes_{\mathrm{min}} B(H) =C^*(\mathbb{F}_\infty) \otimes_{\mathrm{max}} B(H)$. It was left open by Kirchberg  if $B(H) \otimes_{\mathrm{min}} B(H) = B(H) \otimes_{\mathrm{max}} B(H)$, subsequently answered in the negative by Junge--Pisier, \cite{Junge-Pisier:B(H)} and Ozawa--Pisier, \cite{Ozawa-Pisier:B(H)}. He further proved the following fundamental result.

\begin{theorem}[Kirchberg] \label{thm:Kir2}
$C^*(\mathbb{F}_\infty) \otimes_{\mathrm{min}}  C^*(\mathbb{F}_\infty) =C^*(\mathbb{F}_\infty) \otimes_{\mathrm{max}} C^*(\mathbb{F}_\infty)$ if and only if CEP has a positive answer.
\end{theorem}

\noindent Ozawa remarked in \cite[p.\ 23]{Ozawa:CEP} that one can replace the group $\mathbb{F}_\infty$ in Theorem~\ref{thm:Kir2} above with any other group that contains the free group $\mathbb{F}_2$ and whose full group \Cs{} has LLP. This proves that (i) and (ii) are equivalent in the following theorem, since each $C^*(\Gamma_{k,n})$ has the LLP and $\Gamma_{k,n}$ contains $\mathbb{F}_2$ whenever $(k,n) \ne (2,2)$.

\begin{theorem}[Kirchberg, Fritz, Junge et al, Ozawa] \label{thm:equiv}
The following statements are equivalent:
\begin{enumerate}
\item CEP has a positive answer,
\item $C^*(\Gamma_{k,n}) \otimes_{\mathrm{min}}   C^*(\Gamma_{k,n}) = C^*(\Gamma_{k,n}) \otimes_{\mathrm{max}}   C^*(\Gamma_{k,n})$, for all $k,n \ge 2$,
\item Tsirelson's conjecture is true.
\end{enumerate}
\end{theorem}

\noindent The implication (iii) $\Rightarrow$ (ii) was proved by Ozawa in \cite{Ozawa:CEP}.

The breakthrough in the study of Tsirelson's conjecture is due to Slofstra, \cite{Slofstra:Tsirelson}, who proved that $C_{qs}(n,k)$ is not closed, for $n,k$ sufficiently large. This was further refined by Dykema--Paulsen--Prakash, \cite{DykPauPra:non-closure}, who established the result for all $n \ge 5$ and $k \ge 2$. A streamlined and more direct proof was given soon thereafter in \cite{MusRor:Infdim} using that $\cD_\fin(n)$ is not closed, for $n \ge 5$. Both papers \cite{DykPauPra:non-closure} and \cite{MusRor:Infdim} prove, in fact, that the synchronous version $C_q^s(n,2)$ of the set $C_q(n,2)=C^{\mathrm{fin}}_{qs}(n,2)$ is not closed, when $n \ge 5$. Recall that a correlation $[p(a,b \, | \, x,y)]_{a,b;x,y}$ is synchronous if, for every $1 \le x \le n$, $p(a,b \, | \, x,x)=0$ whenever 
$a \ne b$. It was shown by Kim--Paulsen--Schafhauser, \cite[Theorem 3.10]{KPS:sync} that $C_q^s(n,2) = C_{qs}^s(n,2)$, for all $n \ge 2$. (That is, as stated in \cite{KPS:sync}, any synchronous correlation that can be obtained using a tensor product of possibly infinite-dimensional Hilbert spaces, has a representation using only finite-dimensional spaces.) We infer that $C_{qs}(n,2)$ is not closed, for $n \ge 5$, as the synchronous correlations are closed therein.

In \cite[Corollary 3.8]{KPS:sync}, it is further shown that CEP has a positive answer if and only if the \emph{synchronous} Tsirelson conjecture is true, that is,
the closure of $C_{qs}^s(n,k)$ is equal to $C_{qc}^s(n,k)$, for all $n,k\ge 2$. 

In the paper MIP$^*$=RE, \cite{JNVWY:MIP*=RE}, by proving that the complexity class MIP$^*$ contains an undecidable language, the authors Ji, Natarajan, Vidick, Wright and Yuen conclude that the synchronous Tsirelson conjecture is false, thus yielding a negative answer to the Connes Embedding Problem.

\providecommand{\bysame}{\leavevmode\hbox to3em{\hrulefill}\thinspace}
\providecommand{\MR}{\relax\ifhmode\unskip\space\fi MR }
\providecommand{\MRhref}[2]{%
  \href{http://www.ams.org/mathscinet-getitem?mr=#1}{#2}
}
\providecommand{\href}[2]{#2}

\bigskip \noindent
Magdalena Musat\\
Department of Mathematical Sciences \\
University of Copenhagen\\ 
Universitetsparken 5, DK-2100, Copenhagen \O \\
Denmark

\bigskip \noindent Email: musat@math.ku.dk


\end{document}